\documentclass[leqno,11pt,a4paper]{amsart}
\usepackage{pdfsync}
\usepackage[usenames,dvipsnames]{color}
\usepackage{hyperref}
\usepackage{ragged2e}
\usepackage{multirow}
\newtheorem{thm}{Theorem}[section]
\newtheorem{lem}[thm]{Lemma}
\newtheorem{cor}[thm]{Corollary}
\newtheorem{prop}[thm]{Proposition}

\theoremstyle{definition}
\newtheorem{rem}[thm]{Remark}
\theoremstyle{definition}
\newtheorem{eg}[thm]{Example}
\theoremstyle{definition}
\newtheorem{defn}[thm]{Definition}
\numberwithin{equation}{section}
\newenvironment{acknowledge}{\bigskip\noindent\textbf{Acknowledgments.}}{}
\newcommand{\Z}{\mathbb{Z}}
\newcommand{\R}{\mathbb{R}}
\newcommand{\N}{\mathbb{Z}_{\ge 0}}
\newcommand{\OO}{\mathcal{O}}
\newcommand{\abs}[1]{\left\vert{#1}\right\vert}
\newcommand{\vol}[1]{\mathrm{vol}\!\left(#1\right)}
\newcommand{\floor}[1]{\left\lfloor{#1}\right\rfloor}
\newcommand{\roof}[1]{\left\lceil{#1}\right\rceil}
\renewcommand{\det}[1]{\mathrm{det}\!\left(#1\right)}
\begin{document}
\author[G.~Heged{\"u}s]{G{\'a}bor Heged{\"u}s}
\address{Johann Radon Institute for Computational and Applied Mathematics\\Austrian Academy of Sciences\\Altenbergerstra{\ss}e 69\\A-$4040$ Linz\\Austria}
\email{gabor.hegedues@oeaw.ac.at}
\author[A.~M.~Kasprzyk]{Alexander M.~Kasprzyk}
\address{Department of Mathematics\\Imperial College London\\London\ SW7 2AZ\\United Kingdom}
\email{a.m.kasprzyk@imperial.ac.uk}
\title{The boundary volume of a lattice polytope}
\begin{abstract}
For a $d$-dimensional convex lattice polytope $P$, a formula for the boundary volume $\vol{\partial P}$ is derived in terms of the number of boundary lattice points on the first $\floor{d/2}$ dilations of $P$. As an application we give a necessary and sufficient condition for a polytope to be reflexive, and derive formulae for the $f$-vector of a smooth polytope in dimensions $3$, $4$, and $5$. We also give applications to reflexive order polytopes, and to the Birkhoff polytope.
\end{abstract}
\maketitle
\section{Introduction}\label{sec:introduction}
A \emph{lattice polytope} $P\subset \R^d$ is the convex hull of finitely many points in $\Z^d$. We shall assume throughout that $P$ is of maximum dimension, so that $\dim{P}=d$. Denote the boundary of $P$ by $\partial P$. The boundary volume $\vol{\partial P}$ is the volume of each facet of $P$ normalised with respect to the sublattice containing that facet, i.e.
$$\vol{\partial P}:=\sum_{F\text{ facet of }P}\frac{\mathrm{vol}_{d-1}(F)}{\det{\mathrm{aff}\,F\cap\Z^d}},$$
where $\mathrm{vol}_{d-1}(F)$ denotes the $(d-1)$-dimensional volume, and $\det{\mathrm{aff}\,F\cap\Z^d}$ is the determinant of the sublattice contained in the affine hull of $F$.

In two dimensions, the number of lattice points on the boundary of $P$ is equal to the boundary volume. In three dimensions there is a well-known relationship which can be derived directly from Euler's formula and Pick's Theorem (see, for example,~\cite[Proposition 10.2.3]{KasPhD}):

\begin{prop}\label{prop:dim_three_boundary_volume}
Let $P$ be a three-dimensional convex lattice polytope. Then
$$\vol{\partial P}=\abs{\partial P\cap\Z^3}-2.$$
\end{prop}

We shall prove the following generalisation to arbitrary dimension:

\begin{thm}\label{thm:general_boundary_volume}
Let $P$ be a $d$-dimensional convex lattice polytope. Then
\begin{align}
\vol{\partial P}&=\frac{\det{\mathcal{A}}}{\det{\mathcal{D}}}\label{eq:matrix_boundary_volume}\\
&=\frac{1}{(d-1)!}\sum_{m=0}^n(-1)^{n+m}\left({d-1\choose n-m}+(-1)^{d-1}{d-1\choose n+m}\right)\abs{\partial(mP)\cap\Z^d},\label{eq:explicit_boundary_volume}
\end{align}
where $n:=\floor{d/2}$, $\abs{\partial(0P)\cap\Z^d}:=1$, and $\mathcal{A}$ and $\mathcal{D}$ are invertible $n\times n$ matrices defined by
\begin{align*}
\mathcal{A}_{ij}&:=\left\{\begin{array}{ll}
\abs{\partial(iP)\cap\Z^d}-2(d-2n),&\text{if }j=1,\\
i^{d-2j+1},&\text{otherwise};
\end{array}\right.\\
\mathcal{D}_{ij}&:=i^{d-2j+1}.
\end{align*}
\end{thm}

The boundary volume formula for each dimension $4\le d\le 10$ are listed in Table~\ref{tab:boundary_volume}.

\section{A General Boundary Volume Formula}\label{sec:main}
Let $L_P(m):=\abs{mP\cap\Z^d}$ denote the number of lattice points in $P$ dilated by a factor of $m\in\Z_{\geq 0}$. Similarly, let $L_{\partial P}(m):=\abs{\partial(mP)\cap\Z^d}$ denote the number of lattice points on the boundary of $mP$. In two dimensions the relationship between $L_P$ and $L_{\partial P}$ is given by Pick's Theorem~\cite{Pick}. In three dimensions Reeve proved an analogous result:

\begin{thm}[\protect{\cite[Theorem~1]{Ree57}}]\label{thm:reeve_3D}
Let $P$ be a three-dimensional convex lattice polytope. Then
$$2(m-1)m(m+1)\vol{P}=2(L_P(m)-m\abs{P\cap\Z^3})-(L_{\partial P}(m)-m\abs{\partial P \cap\Z^3}),$$
and
$$L_{\partial P}(m)=2(1-m^2) + m^2\abs{\partial P\cap\Z^3}.$$
\end{thm}

In general the function $L_P$ is a polynomial of degree $d$, and is called the \emph{Ehrhart polynomial}. Ehrhart showed that certain coefficients of $L_P$ have natural interpretations in terms of $P$.

\begin{thm}[\cite{Ehr67}]\label{thm:Ehrhart_known_coefficients}
Let $P$ be a $d$-dimensional convex lattice polytope with Ehrhart polynomial $L_P(m)=c_dm^d+\ldots+c_1m+c_0$. Then:
\begin{itemize}
\item[(i)] $c_d=\vol{P}$;
\item[(ii)] $c_{d-1}=(1/2)\vol{\partial P}$;
\item[(iii)] $c_0=1$.
\end{itemize}
\end{thm}

The values of the remaining coefficients of $L_P$ have been studied in, for example,~\cite{Pomm93, Diaz96, BLDPS05}. Particular attention has been paid to the connection with toric geometry; under some additional assumptions, the function $L_P(m)$ calculates $h^0(-mK)$.

Let $P^\circ$ denote the strict interior of $P$. Ehrhart conjectured, and Macdonald proved, a remarkable reciprocity formula connecting $L_P(m)$ and $L_{P^\circ}(m)$ (see~\cite{Dan78} for a proof exploiting Serre--Grothendieck duality).

\begin{thm}[\cite{Mac71}]\label{thm:Ehrhart-Macdonald_reciprocity}
Let $P$ be a $d$-dimensional convex lattice polytope. Then
$$L_P(-m)=(-1)^dL_{P^\circ}(m).$$
\end{thm}

Since $L_P(m)=L_{\partial P}(m)+L_{P^\circ}(m)$ we have the following immediate corollary:

\begin{cor}\label{cor:system_of_equations}
Let $P$ be a $d$-dimensional convex lattice polytope. The coefficients $c_{d-1}$, $c_{d-3}$, $c_{d-5}$, $\ldots$ of $L_P$ satisfy the system of equations:
$$\frac{1}{2}L_{\partial P}(m)=\sum_{i=1}^{\roof{d/2}}m^{d-2i+1}c_{d-2i+1},\quad\text{ for all }m\in\Z_{>0}.$$
\end{cor}

A formula for the volume of an even-dimensional convex lattice polytope was derived by Macdonald in~\cite{Mac63}:
\begin{align*}
\vol{P}=\frac{1}{d!}\Bigg(\sum_{m=1}^{d/2}(-1)^{d/2-m}{d\choose d/2-m}\Big(2&\abs{(mP)^\circ\cap\Z^d}+\\
&\abs{\partial(mP)\cap\Z^d}\Big)+(-1)^{d/2}{d\choose d/2}\Bigg).
\end{align*}
Ko{\l}odziejczyk was able to compute the odd-dimensional formula in~\cite{Kol00}:
$$\vol{P}=\frac{1}{(d+1)!}\sum_{m=1}^{(d+1)/2}(-1)^{(d+1)/2-m}{d+1\choose (d+1)/2-m}m\left(2\abs{(mP)^\circ\cap\Z^d}+\abs{\partial(mP)\cap\Z^d}\right).$$
It is worth noticing that, with a little rearranging, one can combine these results to give a general form remarkably similar to equation~\eqref{eq:explicit_boundary_volume}.

\begin{thm}\label{thm:general_volume}
Let $P$ be a $d$-dimensional lattice polytope. Then
$$\vol{P}=\frac{1}{d!}\sum_{m=0}^N(-1)^{N+m}\left({d\choose N-m}+(-1)^d{d\choose N+m}\right)\left(\abs{mP\cap\Z^d}-\frac{1}{2}\abs{\partial(mP)\cap\Z^d}\right),$$
where $N:=\roof{d/2}$ and $\abs{\partial(0P)\cap\Z^d}:=1$.
\end{thm}

\begin{table}[t]
\begin{center}
\begin{tabular}{|r|l|}
\hline
$d$&$(d-1)!\,\vol{\partial P}$\\
\hline
$4$&$L_{\partial P}(2)-2L_{\partial P}(1)$\\
$5$&$2\left(L_{\partial P}(2)-4L_{\partial P}(1)+6\right)$\\
$6$&$L_{\partial P}(3)-4L_{\partial P}(2)+5L_{\partial P}(1)$\\
$7$&$2\left(L_{\partial P}(3)-6L_{\partial P}(2)+15L_{\partial P}(1)-20\right)$\\
$8$&$L_{\partial P}(4)-6L_{\partial P}(3)+14L_{\partial P}(2)-14L_{\partial P}(1)$\\
$9$&$2\left(L_{\partial P}(4)-8L_{\partial P}(3)+28L_{\partial P}(2)-56L_{\partial P}(1)+70\right)$\\
$10$&$L_{\partial P}(5)-8L_{\partial P}(4)+27L_{\partial P}(3)-48L_{\partial P}(2)+42L_{\partial P}(1)$\\
\hline
\end{tabular}
\vspace{1em}
\end{center}
\caption{The relationship between the boundary volume and the number of boundary points, for each dimension $4\le d\le10$ (see Theorem~\ref{thm:general_boundary_volume}).}
\label{tab:boundary_volume}
\end{table}

\begin{proof}[Proof of Theorem~\ref{thm:general_boundary_volume}]
We wish to express the value of the penultimate coefficient $c_{d-1}$ of $L_P$ in terms of $L_{\partial P}$. A formula for $\vol{\partial P}$ follows from Theorem~\ref{thm:Ehrhart_known_coefficients}~(ii). We shall handle the even dimensional and odd dimensional cases separately. For brevity let us define
$$b_m:=\frac{1}{2m}L_{\partial P}(m),\quad\text{ for all }m\in\Z_{>0}.$$

When $d=2n$ is even, Corollary~\ref{cor:system_of_equations} tells us that the coefficients satisfy the linear system
\begin{equation}\label{eq:system_even}
\begin{pmatrix}
1&1&\ldots&1\\
1&2^2&\ldots&2^{d-2}\\
\vdots&\vdots&&\vdots\\
1&n^2&\ldots&n^{d-2}
\end{pmatrix}\cdot\begin{pmatrix}c_1\\c_3\\\vdots\\c_{d-1}\end{pmatrix}=\begin{pmatrix}b_1\\b_2\\\vdots\\b_n\end{pmatrix}.
\end{equation}
Equation~\eqref{eq:matrix_boundary_volume} follows from an application of Cramer's rule and some elementary matrix operations.

To obtain the explicit description~\eqref{eq:explicit_boundary_volume}, consider the square matrix on the left hand side of~\eqref{eq:system_even}. This is a Vandermonde matrix; we can express its inverse in terms of the product $U\cdot L$ (\cite[equations~(5) and~(7)]{Tur66}), where $U$ is an upper triangular matrix with $1$s on the diagonal, and $L$ is a lower triangular matrix given by
$$L_{ij}=\left\{\begin{array}{ll}
0,&\text{if }i<j,\vspace{0.4em}\\
1,&\text{if }i=j=1,\vspace{0.4em}\\
\displaystyle
\mathop{\prod_{k=1}^i}_{k\ne j}\frac{1}{j^2-k^2},&\text{otherwise.}
\end{array}\right.$$
More explicitly,
$$\begin{pmatrix}c_1\\c_3\\\vdots\\c_{d-1}\end{pmatrix}=
\begin{pmatrix}
\begin{matrix}1&\\&1\end{matrix}&\text{\huge$\star$}\\
\text{\huge$0$}&\begin{matrix}\ddots&\\&1\end{matrix}
\end{pmatrix}\cdot
\begin{pmatrix}
\begin{matrix}1&\\-\frac{1}{3}&\frac{1}{3}\end{matrix}&\text{\huge$0$}\\
\begin{matrix}\vdots&\vdots\\L_{n1}&L_{n2}\end{matrix}&\begin{matrix}\ddots&\\\hdots&L_{nn}\end{matrix}
\end{pmatrix}\cdot
\begin{pmatrix}b_1\\b_2\\\vdots\\b_n\end{pmatrix}.$$
Since we need only know the bottom row of $L$ in order to determine the coefficient $c_{d-1}$, we obtain
\begin{align*}
c_{d-1}&=\sum_{m=1}^n\Big(\mathop{\prod_{k=1}^n}_{k\ne m}\frac{1}{m^2-k^2}\Big)b_m\\
&=2\sum_{m=1}^n{\frac{(-1)^{n+m}m^2}{(n+m)!(n-m)!}b_m}\\
&=\frac{1}{(2n)!}\sum_{m=1}^n(-1)^{n+m}{2n\choose n+m}mL_{\partial P}(m).
\end{align*}
Observing that
$$\frac{m}{n}{2n\choose n+m}={2n-1\choose n-m}-{2n-1\choose n+m}$$
we obtain the result in the even-dimensional case:
$$c_{d-1}=\frac{1}{2\cdot(2n-1)!}\sum_{m=0}^n(-1)^{n+m}\left({2n-1\choose n-m}-{2n-1\choose n+m}\right)L_{\partial P}(m).$$

When $d=2n+1$ is odd we obtain the linear system
$$\begin{pmatrix}
1&1&\ldots&1\\
1&2^2&\ldots&2^{d-3}\\
\vdots&\vdots&&\vdots\\
1&n^2&\ldots&n^{d-3}
\end{pmatrix}\cdot\begin{pmatrix}c_2\\c_4\\\vdots\\c_{d-1}\end{pmatrix}=\begin{pmatrix}b_1-1\\b_2/2-1/2^2\\\vdots\\b_n/n-1/n^2\end{pmatrix}.$$
Once again, Cramer's rule yields~\eqref{eq:matrix_boundary_volume}.

Solving as in the even case, we have that
$$c_{d-1}=\frac{1}{(2n)!}\sum_{m=1}^n(-1)^{n+m}{2n\choose n+m}\left(L_{\partial P}(m)-2\right).$$
From the identity
$$\sum_{m=0}^{2n}(-1)^m{2n\choose m}=0$$
we deduce that
$$2\sum_{m=1}^n(-1)^{n+m}{2n\choose n+m}=(-1)^{n+1}{2n\choose n}.$$
Hence:
\begin{align*}
c_{d-1}&=\frac{1}{(2n)!}\left((-1)^n{2n \choose n}+\sum_{m=1}^n(-1)^{n+m}{2n\choose n+m}L_{\partial P}(m)\right)\\
&=\frac{1}{2\cdot(2n)!}\left(\sum_{m=0}^n(-1)^{n+m}\left({2n\choose n-m}+{2n\choose n+m}\right)L_{\partial P}(m)\right).
\end{align*}
This gives us~\eqref{eq:explicit_boundary_volume}.
\end{proof}
\section{Applications to Reflexive Polytopes}\label{sec:applications_to_reflexive}
In~\cite{Stan80} Stanley proved that the generating function for $L_P$ can be written as a rational function
$$\mathrm{Ehr}_P(t):=\sum_{m\geq 0}L_P(m)t^m=\frac{\delta_0+\delta_1t+\ldots+\delta_dt^d}{(1-t)^{d+1}},$$
where the coefficients $\delta_i$ are non-negative. The sequence $\left(\delta_0,\delta_1,\ldots,\delta_d\right)$ is known as the \emph{$\delta$-vector} of $P$. For an elementary proof of this and other relevant results, see~\cite{BS07} and~\cite{BR}.

The following corollary is a consequence of Theorem~\ref{thm:Ehrhart_known_coefficients}. 

\begin{cor}\label{cor:known_delta_coefficients}
Let $P$ be a $d$-dimension convex lattice polytope with $\delta$-vector $\left(\delta_0,\delta_1,\ldots,\delta_d\right)$. Then:
\begin{itemize}
\item[(i)] $\delta_0=1$;
\item[(ii)] $\delta_1=\abs{P\cap\Z^d}-d-1$;
\item[(iii)] $\delta_d=\abs{P^\circ\cap\Z^d}$;
\item[(iv)] $\delta_0+\ldots+\delta_d=d!\,\vol{P}$.
\end{itemize} 
\end{cor}

Hibi proved~\cite{Hib94} the following lower bound on the $\delta_i$, commonly referred to as the \emph{Lower Bound Theorem}:

\begin{thm}\label{thm:lower_bound_thm}
Let $P$ be a $d$-dimensional convex lattice polytope with $\abs{P^\circ\cap\Z^d}>0$. Then $\delta_1\le\delta_i$ for every $2\le i\le d-1$.
\end{thm}

As a consequence of the Lower Bound Theorem we have a bound on the number of lattice points in $P$ in terms of the volume of $P$. Note that this bound is sharp: equality is given in each dimension by the $d$-simplex $\mathrm{conv}\{e_1,\ldots,e_d,-e_1-\ldots-e_d\}$, where $e_1,\ldots,e_d$ is a basis of $\Z^d$.

\begin{cor}\label{cor:volume_bound}
Let $P$ be a $d$-dimensional convex lattice polytope with $\abs{P^\circ\cap\Z^d}>0$. Then
$$d!\,\vol{P}\ge(d-1)\abs{P\cap\Z^d}-d^2+3.$$
We have equality if and only if the $\delta$-vector of $P$ equals
$$(1,\abs{P\cap\Z^d}-d-1,\abs{P\cap\Z^d}-d-1,\ldots,\abs{P\cap\Z^d}-d-1,1).$$
\end{cor}
\begin{proof}
This is a consequence of Corollary~\ref{cor:known_delta_coefficients} parts~(ii) and~(iv), and Theorem~\ref{thm:lower_bound_thm}.
\end{proof}

A convex lattice polytope $P$ is called \emph{Fano} if $P^\circ\cap\Z^d=\{0\}$; i.e.~if the origin is the only interior lattice point of $P$. A convex lattice polytope $P$ is called \emph{reflexive} if the dual polyhedron
$$P^{\vee}:=\{u\in\R^d\mid \left<u,v\right>\le 1\text{ for all }v\in P\}$$
is also a lattice polytope. Clearly any reflexive polytope is Fano. Reflexive polytopes are of particular relevance to toric geometry: they correspond to Gorenstein toric Fano varieties (see~\cite{Bat94}). There are many interesting characterisations of reflexive polytopes (for example the list in~\cite{HM04}).

\begin{thm}\label{thm:Gorenstein_conditions}
Let $P$ be a $d$-dimensional Fano polytope. The following are equivalent:
\begin{itemize}
\item[(i)] $P$ is reflexive;
\item[(ii)] $L_P(m)=L_{\partial P}(m) + L_P(m-1)$ for all $m\in\Z_{>0}$;
\item[(iii)] $d\,\vol{P}=\vol{\partial P}$;
\item[(iv)] $\delta_i=\delta_{d-i}$ for all $0\le i\le d$.
\end{itemize}
\end{thm}

Theorem~\ref{thm:Gorenstein_conditions}~(iv) is commonly known as \emph{Hibi's Palindromic Theorem}~\cite{Hib91} and can be generalised to duals of non-reflexive polytopes~\cite{FK08}. It is a consequence of a more general result of Stanley's~\cite{Sta78} concerning Gorenstein rings. Clearly any polytope giving equality in Corollary~\ref{cor:volume_bound} must be reflexive.

\begin{rem}\label{rem:volume_for_reflexive}
Of course, as a consequence of equation~\eqref{eq:explicit_boundary_volume} and Theorem~\ref{thm:Gorenstein_conditions}~(iii), one can add the equivalent condition:
\begin{itemize}
\item[(v)] $\displaystyle\vol{P}=\frac{1}{d!}\sum_{m=0}^n(-1)^{n+m}\left({d-1\choose n-m}+(-1)^{d-1}{d-1\choose n+m}\right)\abs{\partial(mP)\cap\Z^d}.$
\end{itemize}
\end{rem}

We are now in a position to express the volume of a reflexive polytope in terms of the number of lattice points in the first $n$ dilations.

\begin{cor}\label{cor:volume_for_reflexive}
Let $P$ be a $d$-dimensional reflexive polytope. Then 
\begin{equation}\label{eq:volume_for_reflexive}
\vol{P}=\frac{1}{d!}\sum_{m=0}^n(-1)^{n+m}\left({d\choose n-m}+(-1)^{d-1}{d\choose n+m+1}\right)\abs{mP\cap\Z^d},
\end{equation}
where $n:=\floor{d/2}$.
\end{cor}
\begin{proof}
This follows from Theorem~\ref{thm:Gorenstein_conditions}~(ii), Remark~\ref{rem:volume_for_reflexive}, and the recursive definition of the binomial coefficient.
\end{proof}

It is tempting to conjecture that the converse of Corollary~\ref{cor:volume_for_reflexive} is true. However, suppose that $P$ is a three-dimensional convex lattice polytope satisfying equation~\eqref{eq:volume_for_reflexive}. By Theorem~\ref{thm:reeve_3D} we have that:
\begin{align*}
L_P(m)-L_{\partial P}(m)-L_P(m-1)=&(\abs{P\cap\Z^3}-\abs{\partial P\cap\Z^3}-1)m^2-\\
&(\abs{P\cap\Z^3}-\abs{\partial P\cap\Z^3}-1)m+(\abs{P\cap\Z^3}-\abs{\partial P\cap\Z^3}-1).
\end{align*}
Thus we require the additional assumption that $\abs{P^\circ\cap\Z^3}=1$; only then would it follow (by Theorem~\ref{thm:Gorenstein_conditions}~(ii)) that $P$ is reflexive.

More generally we can make use of Theorems~\ref{thm:general_boundary_volume} and~\ref{thm:general_volume} to write down a necessary and sufficient relation between the number of points in, and on the boundary of, the first $N$ dilations of $P$.

\begin{thm}\label{thm:New_Gorenstein_condition}
Let $P$ be  $d$-dimensional Fano polytope. $P$ is reflexive if and only if
$$
0=\left\{\begin{array}{ll}
\displaystyle\sum_{m=0}^N(-1)^{N+m}{2N\choose N+m}\left(d\abs{mP\cap\Z^d}-(N+m)\abs{\partial(mP)\cap\Z^d}\right),&\text{if $d$ is even,}\vspace{0.4em}\\
\multicolumn{2}{l}{\displaystyle\sum_{m=0}^N(-1)^{N+m}{2N\choose N+m}\left(md\abs{mP\cap\Z^d}+\left(N^2-m^2-\frac{md}{2}\right)\abs{\partial(mP)\cap\Z^d}\right),}\\
&\text{if $d$ is odd,}
\end{array}\right.
$$
where $N:=\roof{d/2}$ and $\abs{\partial(0P)\cap\Z^d}:=1$.
\end{thm}
\begin{proof}
Suppose first that $d$ is even, so that $N=n$. By Theorem~\ref{thm:Gorenstein_conditions}~(iii), $P$ is reflexive if and only if
\begin{align*}
\sum_{m=0}^n(-1)^{n+m}&\left({d-1\choose n-m}-{d-1\choose n+m}\right)\abs{\partial(mP)\cap\Z^d}=\\
&\sum_{m=0}^n(-1)^{n+m}\left({d\choose n-m}+{d\choose n+m}\right)\left(\abs{mP\cap\Z^d}-\frac{1}{2}\abs{\partial(mP)\cap\Z^d}\right),
\end{align*}
where the left hand side follows from Theorem~\ref{thm:general_boundary_volume}, and the right hand side from Theorem~\ref{thm:general_volume}.

Using the binomial identity
\begin{align*}
{d-1 \choose n-m}-{d-1\choose n+m}=\frac{2m}{d}{d\choose n+m},
\end{align*}
we have that
\begin{align*}
\sum_{m=0}^n(-1)^{n+m}{d\choose n+m}m&\abs{\partial(mP)\cap\Z^d}=\\
&d\sum_{m=0}^n(-1)^{n+m}{d\choose n+m}\left(\abs{mP\cap\Z^d}-\frac{1}{2}\abs{\partial(mP)\cap\Z^d}\right).
\end{align*}
Noticing that $d/2=n$, we obtain our result.

Now suppose that $d$ is odd. In particular, $N=n+1$. In this case we have that $P$ is reflexive if and only if
\begin{align*}
\sum_{m=0}^n(-1)^{n+m}2&{d-1\choose n+m}\abs{\partial(mP)\cap\Z^d}=\\
&\sum_{m=0}^{n+1}(-1)^{n+m}\left({d\choose n+m+1}-{d\choose n+m}\right)\left(\abs{mP\cap\Z^d}-\frac{1}{2}\abs{\partial(mP)\cap\Z^d}\right).
\end{align*}
By standard binomial identities, we have that
\begin{align*}
{d-1\choose n+m}&=\frac{n+m+1}{d}{d\choose n+m+1}\\
&=\frac{n+m+1}{d}{d\choose n-m}\\
&=\frac{(n+m+1)(n-m+1)}{d(d+1)}{d+1\choose n+m+1},
\end{align*}
and that
\begin{align*}
{d\choose n+m+1}-{d\choose n-m+1}=-\frac{2m}{d+1}{d+1\choose n+m+1}.
\end{align*}
Observing that $n-m+1$ vanishes when $m=n+1$, we obtain the equality
\begin{align*}
\sum_{m=0}^{n+1}(-1)^{n+m}&(n+m+1)(n-m+1){d+1\choose n+m+1}\abs{\partial(mP)\cap\Z^d}=\\
&\sum_{m=0}^{n+1}(-1)^{n+m+1}md{d+1\choose n+m+1}\left(\abs{mP\cap\Z^d}-\frac{1}{2}\abs{\partial(mP)\cap\Z^d}\right),
\end{align*}
which is equivalent to
$$\sum_{m=0}^N(-1)^{N+m}{d+1\choose N+m}\left(md\abs{mP\cap\Z^d}+\left((N+m)(N-m)-\frac{md}{2}\right)\abs{\partial(mP)\cap\Z^d}\right)=0.$$
\end{proof}

The conditions given by Theorem~\ref{thm:New_Gorenstein_condition} are summarised in Table~\ref{tab:Gorenstein_condition} for low dimensions.

\begin{table}[t]
\begin{center}
\begin{tabular}{|c|l|}
\hline
$d$&$f(P)$\\
\hline
$3$&$2L_P(2)-L_{\partial P}(2)-4L_P(1)-2L_{\partial P}(1)+8$\\
$4$&$L_P(2)-L_{\partial P}(2)-4L_P(1)+3L_{\partial P}(1)+3$\\
$5$&$2L_P(3)-L_{\partial P}(3)-8L_P(2)+10L_P(1)+11L_{\partial P}(1)-24$\\
$6$&$L_P(3)-L_{\partial P}(3)-6L_P(2)+5L_{\partial P}(2)+15L_P(1)-10L_{\partial P}(1)-10$\\
$7$&$2L_P(4)-L_{\partial P}(4)-12L_P(3)+2L_{\partial P}(3)+28L_P(2)+10L_{\partial P}(2)-28L_P(1)-46L_{\partial P}(1)+80$\\
$8$&$L_P(4)-L_{\partial P}(4)-8L_P(3)+7L_{\partial P}(3)+28L_P(2)-21L_{\partial P}(2)-56L_P(1)+35L_{\partial P}(1)+35$\\
\hline
\end{tabular}
\end{center}
\caption{A $d$-dimensional Fano polytope $P$ is reflexive if and only if the equation $f(P)$ in the second column vanishes (see Theorem~\ref{thm:New_Gorenstein_condition}).}
\label{tab:Gorenstein_condition}
\end{table}

Notice that if $P$ is a reflexive polytope and $d$ is even then, by Theorem~\ref{thm:Gorenstein_conditions}~(ii), Theorem~\ref{thm:New_Gorenstein_condition} reduces to
\begin{align*}
0&=\sum_{m=0}^n(-1)^{n+m}{d\choose n+m}\left(d\abs{mP\cap\Z^d}-(n+m)\left(\abs{mP\cap\Z^d}-\abs{(m-1)P\cap\Z^d}\right)\right)\\
&=\sum_{m=0}^{n-1}(-1)^{n+m}{2n\choose n-m}(n-m)\abs{mP\cap\Z^d}-\\
&\hspace{15em}\sum_{m=0}^{n-1}(-1)^{n+m}{2n\choose n+m+1}(n+m+1)\abs{mP\cap\Z^d}.
\end{align*}
Clearly the right hand side vanishes, so we learn nothing new. The odd-dimensional case is different; the relation is given in Theorem~\ref{thm:New_Gorenstein_relation_odd} and calculated for small dimensions in Table~\ref{tab:New_Gorenstein_relation}.

\begin{thm}\label{thm:New_Gorenstein_relation_odd}
Let $P$ be a reflexive $d$-dimensional polytope, where $d$ is odd. Then
$$\sum_{m=0}^N(-1)^{N+m}{d+2\choose N-m}\abs{mP\cap\Z^d}=0,$$
where $N:=\roof{d/2}$.
\end{thm}
\begin{proof}
From Theorem~\ref{thm:Gorenstein_conditions}~(ii) and Theorem~\ref{thm:New_Gorenstein_condition} we have that
\begin{align*}
0&=(-1)^N{2N\choose N}N^2+\sum_{m=1}^N(-1)^{N+m}{2N\choose N+m}\Bigg(md\abs{mP\cap\Z^d}+\\
&\hspace{7em}\left(N^2-m^2-\frac{md}{2}\right)\left(\abs{mP\cap\Z^d}-\abs{(m-1)P\cap\Z^d}\right)\Bigg)\\
&=\sum_{m=0}^N(-1)^{N+m}{2N\choose N+m}\left(\frac{md}{2}+N^2-m^2\right)\abs{mP\cap\Z^d}-\\
&\hspace{7em}\sum_{m=0}^N(-1)^{N+m}{2N\choose N+m+1}\left(\frac{(m+1)d}{2}-N^2+(m+1)^2\right)\abs{mP\cap\Z^d}.
\end{align*}
Now:
\begin{align*}
{2N\choose N+m}&\left(\frac{md}{2}+N^2-m^2\right)-{2N\choose N+M+1}\left(\frac{(m+1)d}{2}-N^2+(m+1)^2\right)\\
&=\left({2N\choose N+m}-{2N\choose N+m+1}\right)\frac{md}{2}+\\
&\hspace{4em}\left({2N\choose N+m}+{2N\choose N+m+1}\right)(N^2-m^2)-{2N\choose N+m+1}\left(2m+1+\frac{d}{2}\right),
\end{align*}
which, by standard results on the binomial coefficient, reduces to
\begin{align*}
{2N+1\choose N-m}&\frac{md(2m+1)}{2(2N+1)}+{2N+1\choose N-m}(N^2-m^2)-{2N+1\choose N-m}\left(2m+1+\frac{d}{2}\right)\frac{N-m}{2N+1}\\
&={2N+1\choose N-m}\frac{1}{2N+1}\left(\frac{md}{2}(2m+1)+(N-m)(2N^2+N+2mN-m-\frac{d}{2}-1)\right).
\end{align*}
Since $d=2N-1$ we can simplify the term in brackets:
\begin{align*}
\frac{md}{2}(2m+1)&+(N-m)(2N^2+N+2mN-m-\frac{d}{2}-1)\\
&=\frac{md}{2}(2m+1)+(N-m)(2N^2+N+md-\frac{d}{2}-1)\\
&=\frac{md}{2}(2N+1)+(N-m)(2N^2-\frac{1}{2})\\
&=\frac{N(2N-1)(2N+1)}{2}.
\end{align*}
Thus we have that
$$\frac{N(2N-1)}{2}\sum_{m=0}^N(-1)^{N+m}{2N+1\choose N-m}\abs{mP\cap\Z^d}=0.$$
Finally, since we are free to divide through by a non-zero constant, we obtain our result.
\end{proof}

\begin{table}[t]
\begin{center}
\begin{tabular}{|c|l|}
\hline
$d$&$g(P)$\\
\hline
$3$&$L_P(2)-5L_P(1)+10$\\
$5$&$L_P(3)-7L_P(2)+21L_P(1)-35$\\
$7$&$L_P(4)-9L_P(3)+36L_P(2)-84L	_P(1)+126$\\
$9$&$L_P(5)-11L_P(4)+55L_P(3)-165L_P(2)+330L_P(1)-462$\\
\hline
\end{tabular}
\end{center}
\caption{If $P$ is a $d$-dimensional reflexive polytope then the equation $g(P)$ in the second column will vanish (see Theorem~\ref{thm:New_Gorenstein_relation_odd}).}
\label{tab:New_Gorenstein_relation}
\end{table}

By exploiting Hibi's Palindromic Theorem one can express the $\delta_i$ in terms of $L_P(m)$, for $1\le m\le\floor{d/2}$. When $d=4$ we obtain the $\delta$-vector
\begin{equation}\label{rem:Ehr_for_reflexive_4}
(1,\abs{P\cap\Z^4}-5,\abs{2P\cap\Z^4}-5\abs{P\cap\Z^4}+10,\abs{P\cap\Z^4}-5,1),
\end{equation}
and when $d=5$ we have
\begin{equation}\label{rem:Ehr_for_reflexive_5}
(1,\abs{P\cap\Z^5}-6,\abs{2P\cap\Z^5}-6\abs{P\cap\Z^5}+15,\abs{2P\cap\Z^5}-6\abs{P\cap\Z^5}+15,\abs{P\cap\Z^5}-6,1).
\end{equation}

\begin{cor}\label{cor:Reflexive_dilation_bounds}
If $P$ is a $4$-dimensional reflexive polytope then the following bound is sharp:
$$6\abs{P\cap\Z^4}\le\abs{2P\cap\Z^4}+15.$$

If $P$ is a $5$-dimensional reflexive polytope then the following bounds are sharp:
\begin{align*}
\abs{P\cap\Z^5}&\le\frac{1}{7}\abs{2P\cap\Z^5}+3,\\
\abs{2P\cap\Z^5}&\le\frac{1}{4}\abs{3P\cap\Z^5}+7.
\end{align*}
\end{cor}
\begin{proof}
By Theorem~\ref{thm:lower_bound_thm} we have that $\delta_1\le\delta_2$. Applying this to~\eqref{rem:Ehr_for_reflexive_4} gives
\begin{align*}
6\abs{P\cap\Z^4}\le\abs{2P\cap\Z^4}+15,&\text{ when }d=4,\\
\abs{P\cap\Z^5}\le\frac{1}{7}\abs{2P\cap\Z^5}+3,&\text{ when }d=5.
\end{align*}
In the case when $d=5$ we apply Theorem~\ref{thm:New_Gorenstein_relation_odd} to~\eqref{rem:Ehr_for_reflexive_5}, obtaining the second bound.

If $P$ is a $4$-dimensional reflexive polytope such that $6\abs{P\cap\Z^4}=\abs{2P\cap\Z^4}+15$ then it has $\delta$-vector
$$(1,\abs{P\cap\Z^4}-5,\abs{P\cap\Z^4}-5,\abs{P\cap\Z^4}-5,1)$$
and $4!\,\vol{P}=3\abs{P\cap\Z^4}-13$. These conditions are satisfied by the simplex associated with $\mathbb{P}^4$ (see the remark preceding Corollary~\ref{cor:volume_bound}).

Suppose that $P$ is a $5$-dimensional reflexive polytope attaining both of the bounds above. Then
\begin{align*}
\abs{2P\cap\Z^5}&=7\abs{P\cap\Z^5}-21,\\
\text{and }\abs{3P\cap\Z^5}&=28\abs{P\cap\Z^5}-112.
\end{align*}
In particular, the $\delta$-vector is given by
$$(1,\abs{P\cap\Z^5}-6,\abs{P\cap\Z^5}-6,\abs{P\cap\Z^5}-6,\abs{P\cap\Z^5}-6,1),$$
and $5!\,\vol{P}=4\abs{P\cap\Z^5}-22$. An example satisfying these conditions is the simplex associated with $\mathbb{P}^5$.
\end{proof}

The examples given in Corollary~\ref{cor:Reflexive_dilation_bounds} are not unique. A search through {\O}bro's classification of the smooth polytopes in dimensions $4$ and $5$ (which form a subset of the reflexive polytopes) gives many more examples\footnote{\href{http://grdb.lboro.ac.uk/search/toricsmooth?id_cmp=in&id=24,25,127,128,138,139,144,145,147}{\tiny\texttt{http://grdb.lboro.ac.uk/search/toricsmooth?id\_cmp=in\&id=24,25,127,128,138,139,144,145,147}}\\\indent\href{http://grdb.lboro.ac.uk/search/toricsmooth?id_cmp=in&id=148,149,950,954,955,989,990,1008,1009,1010,1013}{\tiny\texttt{http://grdb.lboro.ac.uk/search/toricsmooth?id\_cmp=in\&id=148,149,950,954,955,989,990,1008,1009,1010,1013}}}. These are recorded in Table~\ref{tab:Reflexive_dilation_bounds}.

\begin{table}[t]
\begin{center}
\begin{tabular}{|c|l|}
\hline
$d$&ID\\
\hline
$4$&$24,25,127,128,138,139,144,145,147$\\
\hline
$5$&$148,149,950,954,955,989,990,1008,1009,1010,1013$\\
\hline
\end{tabular}
\vspace{1em}
\end{center}
\caption{The smooth polytopes attaining the bounds in Corollary~\ref{cor:Reflexive_dilation_bounds}. The ID refers to the ID of the polytope in the online \href{http://grdb.lboro.ac.uk/}{Graded Ring Database}; the data was calculated using~\cite{Obr07}.}
\label{tab:Reflexive_dilation_bounds}
\end{table}

\section{Applications to Smooth Polytopes}\label{sec:applications_to_smooth}
Let the number of $i$-dimensional faces of a polytope $P$ be denoted by $f_i$. The vector $(f_0,f_1,\ldots,f_{d-1})$ is called the \emph{$f$-vector} of $P$. By convention $f_{-1}=f_d=1$, representing the empty face $\emptyset$ and the entire polytope $P$. The $f$-vector satisfies \emph{Euler's relation}
\begin{equation}\label{eq:Euler}
\sum_{i=-1}^d(-1)^if_i=0.
\end{equation}
When $P$ is simplicial (i.e. the facets of $P$ are $(d-1)$-simplicies) the \emph{Dehn-Sommerville equations} give some additional relations amongst the $f_i$. Conjectured by Dehn and first proved by Sommerville, these equations did not become widely known until they were rediscovered by Klee.

\begin{thm}[\cite{Klee64}]\label{thm:Dehn-Sommerville}
Let $P$ be a $d$-dimensional simplicial lattice polytope with $f$-vector $(f_0,f_1,\ldots,f_{d-1})$. Then
$$f_i=\sum_{j=i}^{d-1}(-1)^{d-1-j}{j+1\choose i+1}f_j,\quad\text{ for }1\le i\le d-2.$$
\end{thm}

A $d$-dimensional convex lattice polytope $P$ is called \emph{smooth} if the vertices of any facet of $P$ form a $\Z$-basis of the ambient lattice $\Z^d$. Any such $P$ is simplicial and reflexive. Smooth polytopes are in bijective correspondence with smooth toric Fano varieties, and as such have been the subject of much study (see, for example,~\cite{Baty91,Obr07}).

In~\cite{ParkH03} Park investigated the $f$-vector of smooth polytopes of dimension $3\le d\le 5$ and established weak bounds on the $f_i$. We shall make use of Theorem~\ref{thm:general_boundary_volume} to give an explicit description of the $f$-vector in those dimensions.

\begin{thm}\label{thm:Smooth_f-vectors}
If $P$ is a $3$-dimensional smooth polytope then its $f$-vector is given by
$$\big(\abs{\partial P\cap\Z^3},3\abs{\partial P\cap\Z^3}-6,2\abs{\partial P\cap\Z^3}-4\big).$$
If $P$ is a $4$-dimensional smooth polytope then its $f$-vector is given by
\begin{align*}
\big(\abs{\partial P\cap\Z^4},\abs{\partial(2P)\cap\Z^4}-\abs{\partial P\cap\Z^4},2\abs{\partial(2P)\cap\Z^4}-4\abs{\partial P\cap\Z^4},&\\
\abs{\partial(2P)\cap\Z^4}-&2\abs{\partial P\cap\Z^4}\big).
\end{align*}
If $P$ is a $5$-dimensional smooth polytope then its $f$-vector is given by
\begin{align*}
\big(\abs{\partial P\cap\Z^5},\abs{\partial(2P)\cap\Z^5}-\abs{\partial P\cap\Z^5},4\abs{\partial(2P)\cap\Z^5}-14\abs{\partial P\cap\Z^5}+20,&\\
5\abs{\partial(2P)\cap\Z^5}-20\abs{\partial P\cap\Z^5}+30,2\abs{\partial(2P)\cap\Z^5}-&8\abs{\partial P\cap\Z^5}+12\big).
\end{align*}
\end{thm}
\begin{proof}
Let $P$ be a $d$-dimensional smooth polytope. By definition each facet $F$ of $P$ is a simplex whose vertices generate the underlying lattice $\Z^d$. Hence $\vol{F}=1/(d-1)!$, so
$$(d-1)!\,\vol{\partial P}=f_{d-1}.$$
Furthermore, $\abs{\partial P\cap \Z^n}=f_0$.
\begin{itemize}
\item[$d=3$:]
Theorem~\ref{thm:Dehn-Sommerville} gives $2f_1=3f_2$, and Theorem~\ref{thm:general_boundary_volume} yields $f_2=2f_0-4$. Thus the $f$-vector is uniquely defined in terms of $f_0$.
\item[$d=4$:]
In this case Theorem~\ref{thm:Dehn-Sommerville} gives $f_2=2f_3$. Applying~\eqref{eq:Euler} we obtain $f_1=f_0+f_3$. Finally, Theorem~\ref{thm:general_boundary_volume} tells us that $f_3=\abs{\partial(2P)\cap\Z^4}-2f_0$. The result follows.
\item[$d=5$:]
In dimension five Theorem~\ref{thm:Dehn-Sommerville} and equation~\eqref{eq:Euler} give three relations:
\begin{align*}
2f_1&=3f_2-5f_4,\\
2f_3&=5f_4,\\
2f_0-f_2+2f_4&=4.
\end{align*}
From Theorem~\ref{thm:general_boundary_volume} we know that $f_4=2\abs{\partial(2P)\cap\Z^5}-8f_0+12$. Substituting, we see that the $f$-vector is uniquely defined in terms of $\abs{\partial(2P)\cap\Z^5}$ and $\abs{\partial P\cap\Z^5}$.
\end{itemize}
\end{proof}

It is worth noting that Casagrande~\cite{Cas04} proves a sharp bound for $\abs{\partial P\cap\Z^d}$ in terms of the dimension, and Batyrev~\cite[Theorem~2.3.7]{Bat99} gives us a bound on $f_{d-3}$ in terms of $f_{d-2}$. Bremner and Klee~\cite{BK99} tell us a lower bound on $f_1$ in terms of $f_0$ and $d$. These results are collected in the following theorem.

\begin{thm}\label{thm:f_i_bounds}
Let $P$ be $d$-dimensional smooth polytope. Then the following inequalities hold:
\begin{itemize}
\item[(i)] $\abs{\partial P\cap\Z^d}\le\left\{\begin{array}{ll}
3d,&\text{if $d$ is even;}\\
3d-1,&\text{if $d$ is odd.}
\end{array}\right.$
\item[(ii)] $12f_{d-3}\ge(3d-4)f_{d-2}.$
\item[(iii)] $df_0\leq f_1+{d+1\choose 2}.$
\end{itemize}
\end{thm}

Thus we obtain upper and lower bounds on $\abs{\partial(2P)\cap\Z^d}$ when $d=4$ or $5$.

\begin{cor}\label{cor:2P_bounds}
If $P$ is a $4$-dimensional smooth polytope then
$$5\abs{\partial P\cap\Z^4}-10\le\abs{\partial(2P)\cap\Z^4}\leq 5\abs{\partial P\cap\Z^4}.$$
If $P$ is a $5$-dimensional smooth polytope then
$$42\abs{\partial P\cap\Z^5}-105\le 7\abs{\partial(2P)\cap\Z^5}\le 52\abs{\partial P\cap\Z^5}-90.$$
\end{cor}
\begin{proof}
Apply Theorem~\ref{thm:f_i_bounds}~(ii) and~(iii) to Theorem~\ref{thm:Smooth_f-vectors}.
\end{proof}

\begin{cor}\label{cor:vol_bounds}
If $P$ is a $4$-dimensional smooth polytope then
$$4!\,\vol{P}\leq 3f_0.$$
If $P$ is a $5$-dimensional smooth polytope then
$$5!\,\vol{P}\le\frac{48f_0-96}{7}.$$
\end{cor}
\begin{proof}
Recall that since $P$ is smooth, $d!\,\vol{P}=(d-1)!\,\vol{\partial P}=f_{d-1}$.
In each case Theorem~\ref{thm:Smooth_f-vectors} tells us the value for $f_{d-1}$. Applying Corollary~\ref{cor:vol_bounds} immediately gives the result in dimension four.

In dimension five we see that
\begin{align*}
7\cdot 5!\,\vol{P}&=7\cdot 4!\,\vol{\partial P}\\
&=2(7\abs{\partial(2P)\cap\Z^5}-28\abs{\partial P\cap\Z^5}+42)\\
&\le 2(24f_0-48),
\end{align*}
where the final inequality is an application of Corollary~\ref{cor:vol_bounds}.
\end{proof}

The smooth polytopes attain either the lower or the upper limit in Corollary~\ref{cor:2P_bounds} are listed\footnote{\href{http://grdb.lboro.ac.uk/search/toricsmooth?id_cmp=in&id=24,25,127,128,138,139,144,145,147}{\tiny\texttt{http://grdb.lboro.ac.uk/search/toricsmooth?id\_cmp=in\&id=24,25,127,128,138,139,144,145,147}}\\\indent\href{http://grdb.lboro.ac.uk/search/toricsmooth?id_cmp=in&id=148,149,950,954,955,989,990,1008,1009,1010,1013}{\tiny\texttt{http://grdb.lboro.ac.uk/search/toricsmooth?id\_cmp=in\&id=148,149,950,954,955,989,990,1008,1009,1010,1013}}} in Table~\ref{tab:limiting_ids}. The upper bound in dimension five is not sharp.

\begin{table}[t]
\begin{center}
\begin{tabular}{|c|c|l|}
\hline
$d$&Equality&ID\\
\hline
\multirow{2}{*}{$4$}&Lower&$24,25,127,128,138,139,144,145,147$\\
&Upper&$63,100$\\
\hline
\multirow{2}{*}{$5$}&Lower&$148,149,950,954,955,989,990,1008,1009,1010,1013$\\
&Upper&None\\
\hline
\end{tabular}
\vspace{1em}
\end{center}
\caption{The smooth polytopes attaining one of the bounds in Corollary~\ref{cor:2P_bounds}. The ID refers to the ID of the polytope in the online \href{http://grdb.lboro.ac.uk/}{Graded Ring Database}; the data was calculated using~\cite{Obr07}.}
\label{tab:limiting_ids}
\end{table}
\section{Reflexive order polytopes}\label{sec:order_polytopes}

Throughout let $Q$ be a finite poset with $d:=\abs{Q}$ elements. Let $\Omega(Q,k)$ denote the number of order--preserving maps $f:Q\rightarrow C_k$, where $C_k$ is the chain with $k\in\Z_{>0}$ elements; i.e.~if $x\leq y$ in $Q$, then $f(x)\leq f(y)$. Then $\Omega(Q,k)$ is a polynomial in $k$ of degree $d$, called the \emph{order polynomial} of $Q$. 

Let $\bar{\Omega}(Q,k)$ denote the number of strictly order--preserving maps $f:Q\to C_k$; i.e.~if $x< y$ in $Q$, then $f(x)<f(y)$. Once again $\bar{\Omega}(Q,k)$ is a polynomial in $k$ of degree $d$; it is called the \emph{strict order polynomial} of $Q$. 

\begin{defn}\label{defn:graded_and_rank_posets}
A poset $Q$ is said to be \emph{graded} if there exists an order--preserving function $f$ such that whenever $y$ covers $x$, $f(y)=f(x)+1$. Equivalently, all maximal chains of $Q$ have the same length $r$. Following Stanley~\cite[Chapter 3.1]{Stan97} we shall call $r$ the \emph{rank} of $Q$. In particular one can adjoin a unique minimum element $\hat{0}$ and unique maximum element $\hat{1}$ to $Q$ to obtain a bounded, graded poset $\hat{Q}$ of rank $r+2$.
\end{defn}

\begin{defn}\label{defn:linear_extension}
A bijective order--preserving map is called a \emph{linear extension} of $Q$. The number of linear extensions is denoted by $e(Q)$.
\end{defn}

Let $e_s(Q)$ denote the number of surjective order--preserving maps $f:Q\to C_s$.

\begin{eg}
If $Q$ is the antichain with $\abs{Q}=d$ then $\bar{\Omega}(Q,k)=\Omega(Q,k)=k^d$ and $e(Q)=d!$. If $Q$ is the chain $C_d$ then $\Omega(Q,k)={d+k-1\choose d}$, $\bar{\Omega}(Q,k)={k \choose d}$, and $e(Q)=1$. 
\end{eg}

\begin{thm}[\cite{Stan70}]\label{thm:order_reciprocity}
Let $Q$ be a finite poset with $\abs{Q}=d$ and order polynomial $\Omega(Q,k)=a_dk^d+\ldots+a_1k+a_0$. Then:
\begin{itemize}
\item[(i)] $\bar{\Omega}(Q,k)=(-1)^d\Omega(Q,-k)$ for all $k\in\Z$;
\item[(ii)] If $Q$ is graded of rank $r$, then  $a_{d-1}=\frac{re(Q)}{2(d-1)!}$;
\item[(iii)] If $Q$ is graded of rank $r$, then $\Omega(Q,-r-k)=(-1)^d\Omega(Q,k)$ for each $k\in \Z$;
\item[(iv)] $a_d=\frac{e(Q)}{d!}$.
\item[(v)] $\Omega(Q,k)=\sum_{s=1}^d e_s(Q){k\choose s}$. 
\item[(vi)] If $Q$ is graded of rank $r$, then $2e_{d-1}(Q)=(d-r+1)e(Q)$.
\end{itemize}
\end{thm}

Theorem~\ref{thm:order_reciprocity}~(i) is commonly referred to as the \emph{Reciprocity Theorem for the Order Polynomial}.

\begin{defn}\label{defn:order_polytope}
The \emph{order polytope} $\OO(Q)$ of a poset $Q$ is the set of order-preserving maps from $Q$ to the interval $[0,1]$, i.e.~the set of all functions $f$ satisfying
$$\begin{array}{ll}
0\le f(x)\le 1,&\text{for all }x\in Q;\\
f(x)\le f(y),&\text{if $y$ covers $x$ in Q}.
\end{array}$$
\end{defn}

Given the bounded poset $\hat{Q}$ one can define $\hat{\OO}(Q)$ as the set of all functions $g$ such that
$$\begin{array}{rl}
&g(\hat{0})=0,\\
&g(\hat{1})=1,\\
\text{and }&g(x)\le g(y),\text{ if $y$ covers $x$ in $\hat{Q}$}.
\end{array}$$
Then the bijective linear map $\rho:\hat{\OO}(Q)\rightarrow \OO(Q)$ given by restriction to $Q$ defines a combinatorial equivalence of polytopes. Stanley was able to derive the entire facial structure of $\hat{\OO}(Q)$ (\cite[\S1]{Stan86}). In particular, the number of facets of $\OO(Q)$ is equal to the number of cover relations in $\hat{Q}$, and the number of vertices of $\OO(Q)$ is given by:
$$\abs{\{I\subset Q\mid\text{ if $x\in I$ and $y\ge x$ then $y\in I$}\}}.$$

\begin{thm}[\protect{\cite[\S4]{Stan86}}]\label{thm:order_poly_to_Ehrhart}
Let $Q$ be a finite poset with $\abs{Q}=d$. Then:
\begin{itemize}
\item[(i)] $L_{\OO(Q)}(k)=\abs{k\OO(Q)\cap {\Z}^d}=\Omega(Q,k+1)$ for each $k\in\Z$;
\item[(ii)] $\vol{\OO(Q)}=\frac{e(Q)}{d!}$.
\end{itemize}
\end{thm}

\begin{cor}\label{cor:interior}
Let $Q$ be a finite poset with order polytope $P:=\OO(Q)$. Then
$$\abs{(kP)^\circ\cap\Z^d}=\bar{\Omega}(Q,k-1),\qquad\text{ for all }k\in\Z_{>0}.$$
\end{cor}
\begin{proof} 
This is immediate from Theorems~\ref{thm:Ehrhart-Macdonald_reciprocity},~\ref{thm:order_reciprocity}~(i), and~\ref{thm:order_poly_to_Ehrhart}~(i):
\begin{align*}
L_{P^\circ}(k)&=(-1)^dL_P(-k)\\
&=(-1)^d\Omega(Q,1-k)\\
&=\bar{\Omega}(Q,k-1),\qquad\text{ for any }k\in\Z_{\ge 0}.
\end{align*}
\end{proof}

\begin{rem}\label{rem:unique_interior_dilation}
Suppose that $Q$ is graded of rank $r$. Stanley showed~\cite[Corollary 4.5.17]{Stan97} that
$$\Omega(Q,-1)=\Omega(Q,-2)=\ldots=\Omega(Q,-r)=0,$$
and that
$$\Omega(Q,-r-1)=(-1)^d.$$
From Corollary~\ref{cor:interior} we see that $(r+2)\OO(Q)$ is the smallest integral dilation of $\OO(Q)$ with an interior lattice point; in fact $(r+2)\OO(Q)$ contains a \emph{unique} interior lattice point.
\end{rem}


\begin{prop}\label{prop:bound_volume_arbitrposet}
Let $Q$ be a poset with $\abs{Q}=d$.
Let $P:=\OO(Q)$ be the order polytope of $Q$. Then the boundary volume of $P$ is
$$\vol{\partial P}=\frac{(3-d)e(Q)+2e_{d-1}(Q)}{(d-1)!}.$$
If in addition $Q$ is a graded poset of rank $r$ then the boundary volume of $P$ is
$$\vol{\partial P}=\frac{(r+2)e(Q)}{(d-1)!}.$$
\end{prop}
\begin{proof}
Since $L_{P}(k)=\Omega(Q,k+1)$ for each $k\in\Z$ by Theorem~\ref{thm:order_poly_to_Ehrhart}~(i), hence if we express $\Omega(Q,n)$ as
$$\Omega(Q,n)=\sum_{i=0}^d a_in^i,$$
then
$$L_{P}(n)=\Omega(Q,n+1)=a_d(n+1)^d+a_{d-1}(n+1)^{d-1}+\sum_{i=0}^{d-2} a_i(n+1)^i.$$
If we express $L_P(n)$ in the form
$$L_P(n)=\sum_{i=0}^d c_in^i$$
then $c_{d-1}=a_{d-1}+da_d$.

Using Theorem~\ref{thm:order_reciprocity}~(iv) and ~(v), we get that
$$c_{d-1}=d\frac{e(Q)}{d!}+\frac{e_{d-1}(Q)}{(d-1)!}-\frac{{d\choose 2}e(Q)}{d!}=\frac{(e(Q)(3-d)+2e_{d-1}(Q)}{2(d-1)!}.$$
But it follows from Theorem ~\ref{thm:Ehrhart_known_coefficients}~(ii), that 
$$(1/2)\vol{\partial P}=c_{d-1}=\frac{(e(Q)(3-d)+2e_{d-1}(Q)}{2(d-1)!}.$$
Combining these results gives
$$\vol{\partial P}=\frac{(3-d)e(Q)+2e_{d-1}(Q)}{(d-1)!}.$$

When $Q$ is a graded poset, applying Theorem~\ref{thm:order_reciprocity}~(vi) to the previous formula gives
$$\vol{\partial P}=\frac{(r+2)e(Q)}{(d-1)!}.$$
\end{proof}

\begin{lem}\label{lem:reflexive_order_poly}
Let $Q$ be a graded poset of rank $r$ with $\abs{Q}=d$. Then $(r+2)\OO(Q)$ is a translate of a reflexive polytope.
\end{lem}
\begin{proof}
Let $P:=(r+2)\OO(Q)$ be the $(r+2)$-th dilate of the order polytope $\OO(Q)$ of $Q$. It is enough to prove that $d\,\vol{P}=\vol{\partial P}$. But
\begin{align*}
d\,\vol{P}&=d\,\vol{(r+2)\OO(Q)}\\
&=d(r+2)^d\vol{\OO(Q)}\\
&=d(r+2)^d\frac{e(Q)}{d!}\\
&=(r+2)^{d-1}\frac{(r+2)e(Q)}{(d-1)!}\\
&=(r+2)^{d-1}\vol{\partial \OO(Q)}\\
&=\vol{\partial P}.
\end{align*}
\end{proof}



Since $(r+2)\OO(Q)$ is a (translate of a) reflexive polytope, we can reinterpret our results from Section~\ref{sec:applications_to_reflexive} in terms of the order polytope:

\begin{cor}[c.f.~Corollary~\ref{cor:volume_for_reflexive}]\label{cor:volume_for_reflexive_order_polytope}
Let $Q$ be a finite graded poset of rank $r$ with $\abs{Q}=d$. Let $e(Q)$ denote the number of linear extensions of $Q$. Then
$$(r+2)^de(Q)=\sum_{m=0}^n(-1)^{n+m}\left({d\choose n-m}+(-1)^{d-1}{d\choose n+m+1}\right)\Omega(Q,m(r+2)+1),$$
where $n:=\floor{d/2}$. 
\end{cor}

\begin{thm}[c.f.~Theorem~\ref{thm:New_Gorenstein_relation_odd}]\label{thm:New_Gorenstein_relation_odd_order_polytope}
Let $Q$ be a finite graded poset of rank $r$ with $\abs{Q}=d$. Suppose that $d$ is odd. Then
$$\sum_{m=0}^N (-1)^{N+m}{d+2\choose N-m}\Omega(Q,m(r+2)+1)=0,$$
where $N:=\roof{d/2}$. 
\end{thm}

\section{The Birkhoff polytope}\label{sec:Birkhoff}
Let $B(d)$ denote the \emph{Birkhoff polytope} (or \emph{transportation polytope}) of $d\times d$ doubly stochastic matrices in $\R^{d^2}$. That is, $B(d)$ is defined by
$$x_{i,j}\ge 0,\quad\sum_{i=1}^dx_{i,j}=1,\quad\sum_{j=1}^dx_{i,j}=1,\quad\text{for all }1\leq i,j\leq d.$$

Because of its rich combinatorial properties, the Birkhoff polytope has been intensively studied. In particular, methods for estimating and computing the volume and Ehrhart polynomial are of considerable interest (see~\cite{Pak00,BP03,CM09}). The following theorem summarises some of the key facts about $B(d)$:

\begin{thm} \label{thm:Birkhoff_summary}
Let $B(d)$ denote the polytope of $d\times d$ doubly stochastic matrices in $\R^{d^2}$.
Let $H_n(r)$ denote the number of $n\times n$ magic squares with linear sums equal to $r$. 
Let $P_n(r)$ denote the number of $n\times n$ positive magic squares with linear sums equal to $r$, where positive means that all entries of the matrix are positive. Then:
\begin{itemize}
\item [(i)] $\dim{B(d)}=(d-1)^2$;
\item [(ii)] $L_{B(d)}(m)=H_d(m)$ for all $d\in \Z_{>0}$ and $m\in \N$;
\item [(iii)] $L_{B(d)}(-d-t)=(-1)^{(d-1)^2}L_{B(d)}(t)$ for all $t\in \Z$;
\item [(iv)] the vertices of $B(d)$ are the permutation matrices;
\item [(v)] $L_{B(d)^{\circ}}(m)=P_d(m)$ for all $d\in \Z_{>0}$ and $m\in \Z_{>0}$.
\end{itemize}
\end{thm}

In fact -- as the following two lemmas show -- it is easy to see that the $d$-th dilation of the Birkhoff polytope contains precisely one interior lattice point, and that this dilation is a translate of a reflexive polytope.

\begin{lem}\label{lem:Birk_interior}
Let $B(d)$ denote the polytope of $d\times d$ doubly stochastic matrices in $\R^{d^2}$. Then $\abs{dB(d)^\circ\cap \Z^{d^2}}=1$.
\end{lem}
\begin{proof}
Using Theorem~\ref{thm:Birkhoff_summary}~(v),
$$\abs{dB(d)^\circ\cap \Z^{d^2}}=L_{B(d)^{\circ}}(d)=P_d(d).$$
But if $Q$ is a $d\times d$ positive magic square whose lines sum to $d$, then $Q$ must be the matrix with all entries equal to one. Hence $P_d(d)=1$.
\end{proof}

\begin{lem}\label{lem:Birkreflexive_pol}
Let $P:=dB(d)-Q$ denote the translation of the $d$-th dilate of the Birkhoff polytope by $Q$, where $Q$ is the matrix with all entries equal to one. Then $P$ is a reflexive polytope.
\end{lem}
\begin{proof}
From Theorem~\ref{thm:Gorenstein_conditions}~(i) and~(ii) it is enough to show that 
$$(-1)^{(d-1)^2}L_{dB(d)}(-m)=L_{dB(d)}(m-1)$$
for all $m\in \Z_{>0}$. But setting $t=d(m-1)$ in Theorem~\ref{thm:Birkhoff_summary}~(iii) gives:
$$L_{dB(d)}(m)=L_{B(d)}(dm)=(-1)^{(d-1)^2}L_{B(d)}(d(m-1))=(-1)^{(d-1)^2}L_{dB(d)}(m-1).$$
\end{proof}

We can now reinterpret our results in Section~\ref{sec:applications_to_reflexive} in terms of the Birkhoff polytope. In particular an explicit formula for the volume of the Birkhoff polytope is given in terms of the the first $\floor{(d-1)^2/2}$ dilations.

\begin{cor}[c.f.~Corollary~\ref{cor:volume_bound}]\label{cor:Birk_Ehrhart}
Let $B(d)$ denote the polytope of $d\times d$ doubly stochastic matrices in $\R^{d^2}$. Then 
$$((d-1)^2)!\,d^{(d-1)^2}\vol{B(d)}\geq (d-1)^2H_d(d)-(d-1)^2+3$$
\end{cor}

\begin{cor}[c.f.~Corollary~\ref{cor:volume_for_reflexive}]\label{cor:Birk_volume}
Let $B(d)$ denote the polytope of $d\times d$ doubly stochastic matrices in $\R^{d^2}$. Then 
\begin{equation}\label{eq:volume_for_Birkhoff}
\vol{B(d)}=\frac{1}{((d-1)^2)!d^{(d-1)^2}}\sum_{m=0}^n(-1)^{n+m}\left({(d-1)^2\choose n-m}+(-1)^d{(d-1)^2\choose n+m+1}\right)H_d(md),
\end{equation}
where $n:=\floor{(d-1)^2/2}$.
\end{cor}

\begin{thm}[c.f.~Theorem~\ref{thm:New_Gorenstein_relation_odd}]\label{thm:Birk_equation}
Suppose that $d$ is even. Then
\begin{equation}\label{eq:reflexive_Birkhoff}
\sum_{m=0}^N (-1)^{N+m}{d^2-2d+3\choose N-m}H_d(dm)=0,
\end{equation}
where $N:=\roof{(d-1)^2/2}$. 
\end{thm}

\begin{acknowledge}
The authors would like to thank Josef Schicho and Benjamin Nill for many helpful remarks, and Don Taylor for alerting them to~\cite{Tur66}. This work was completed whilst the second author was a guest at \textsc{Ricam}.
\end{acknowledge}
\bibliographystyle{amsalpha}
\newcommand{\etalchar}[1]{$^{#1}$}
\providecommand{\bysame}{\leavevmode\hbox to3em{\hrulefill}\thinspace}
\providecommand{\MR}{\relax\ifhmode\unskip\space\fi MR }
\providecommand{\MRhref}[2]{%
  \href{http://www.ams.org/mathscinet-getitem?mr=#1}{#2}
}
\providecommand{\href}[2]{#2}

\end{document}